\documentclass[journal]{IEEEtran}

\usepackage{amsmath}    
\usepackage{graphics,epsfig,subfigure}    
\usepackage{verbatim}   
\usepackage{color}      
\usepackage{subfigure}  
\usepackage{amssymb}
\usepackage{epstopdf}
\usepackage{graphicx}
\newcommand{\prob}{Pr}
\newcommand{\expect}[1]{\mathbb{E}\big\{#1\big\}}
\newcommand{\defequiv}{\mbox{\raisebox{-.3ex}{$\overset{\vartriangle}{=}$}}}

\newcommand{\bv}[1]{{\boldsymbol{#1} }}
\newcommand{\script}[1]{{{\cal{#1} }}}

\begin{document}
\title{LIFO-Backpressure Achieves Near Optimal Utility-Delay Tradeoff}

\author{\large{Longbo Huang, Scott Moeller, Michael J. Neely, Bhaskar Krishnamachari}%
\thanks{Longbo Huang, Scott Moeller, Michael J. Neely, and Bhaskar Krishnamachari   (emails: \{longbohu, smoeller, mjneely, bkrishna\}@usc.edu)
are with the Department of Electrical
Engineering, University of Southern California, Los Angeles, CA 90089, USA.}
\thanks{This material is supported in part  under one or more of 
the following grants: DARPA IT-MANET 
W911NF-07-0028, 
NSF CAREER CCF-0747525, and continuing through participation in the 
Network Science Collaborative Technology Alliance sponsored
by the U.S. Army Research Laboratory.} }
\maketitle

\thispagestyle{empty}
\newtheorem{rem}{Remark}
\newtheorem{fact_def}{\textbf{Fact}}
\newtheorem{coro}{\textbf{Corollary}}
\newtheorem{lemma}{\textbf{Lemma}}
\newtheorem{main}{\textbf{Proposition}}
\newtheorem{thm}{\textbf{Theorem}}
\newtheorem{claim}{\emph{Claim}}
\newtheorem{prop}{Proposition}
\newtheorem{assumption}{\textbf{Assumption}}

 \begin{abstract}
There has been considerable recent work developing a new stochastic network utility maximization framework using Backpressure algorithms, also known as MaxWeight.  A key open problem has been the development of utility-optimal algorithms that are also delay efficient.  In this paper, we show that the Backpressure  algorithm, when combined with the LIFO queueing discipline (called LIFO-Backpressure), is able to achieve a utility that is within $O(1/V)$ of the optimal value for any scalar $V\geq1$, while maintaining an average delay of $O([\log(V)]^2)$ for all but a tiny fraction of the network traffic. 
This result holds for general stochastic network optimization problems and  general Markovian dynamics. 
Remarkably, the performance of LIFO-Backpressure can be achieved by simply changing the queueing discipline; it requires no other modifications of the original Backpressure algorithm.
We validate the results through empirical measurements from a sensor network testbed, which show good match between theory and practice.

 \end{abstract}


\section{Introduction}

Recent developments in stochastic network optimization theory have yielded a very general framework that solves a large class of networking problems of the following form: 
We are given a discrete time stochastic network. The network state, which describes current realization of the underlying network randomness, such as the network channel condition, is time varying according to some probability law. A network controller performs some action based on the observed network state at every time slot. The chosen action incurs a cost, \footnote{Since cost minimization is mathematically equivalent to utility maximization, below we will use cost and utility interchangeably} 
  but also serves some amount of traffic and possibly generates new traffic for the network.  This traffic causes congestion, and thus leads to backlogs at nodes in the network. The goal of the controller is to minimize its time average cost subject to the constraint that the time average total backlog in the network be kept finite. 

This general setting models a large class of networking problems ranging from traffic routing \cite{tassiulas92}, 
flow utility maximization \cite{eryilmaz_qbsc_ton07}, 
network pricing \cite{huangneelypricing} to cognitive radio applications \cite{rahulneelycognitive}. Also,  many techniques have also been applied to this problem (see \cite{yichiang_netopt08} for a survey).  
Among the approaches that have been adopted, the family of Backpressure algorithms \cite{neelynowbook} are recently receiving much attention due to their provable performance guarantees, robustness to stochastic network conditions and, most importantly, their ability to achieve the desired performance \emph{without requiring any statistical knowledge} of the underlying randomness in the network.  



Most prior performance results for Backpressure are given in the following $[O(1/V), O(V)]$ utility-delay tradeoff form \cite{neelynowbook}: Backpressure is able to achieve a utility that is within $O(1/V)$ of the optimal utility for any scalar $V\geq1$, while guaranteeing a average network delay that is  $O(V)$. Although these results provide strong theoretical guarantees for the algorithms, the network delay can actually be unsatisfying when we achieve a utility that is very close to the optimal, i.e., when $V$ is large. 

There have been previous works trying to develop algorithms that can achieve better utility-delay tradeoffs. 
Previous works \cite{neelysuperfast} and \cite{neelyenergydelay} show improved tradeoffs are possible
for single-hop networks with certain structure, and develops
optimal $[O(1/V), O(\log(V))]$and $[O(1/V), O(\sqrt{V})]$ utility-delay tradeoffs.  
However, the algorithms are different from basic Backpressure and require knowledge of an ``epsilon'' parameter that measures distance to a performance
region boundary. 
Work  \cite{huangneely_dr_tac} uses a completely different
analytical technique to show that similar poly-logarithmic tradeoffs, i.e., $[O(1/V), O([\log(V)]^2)]$, are possible by carefully modifying the actions taken by the basic Backpressure algorithms. 
However, the algorithm requires a pre-determined learning phase, which adds additional complexity to the implementation. 
The current work, following the line of analysis in \cite{huangneely_dr_tac}, instead shows  that similar poly-logarithmic tradeoffs, i.e., $[O(1/V), O([\log(V)]^2)]$, 
can be achieved by the \emph{original} Backpressure algorithm by simply
modifying the service discipline from First-in-First-Out (FIFO) to 
Last-In-First-Out (LIFO) (called LIFO-Backpressure below).   This is a remarkable feature that distinguishes LIFO-Backpressure from previous algorithms in \cite{neelysuperfast}  \cite{neelyenergydelay} \cite{huangneely_dr_tac}, and 
provides a deeper understanding of 
backpressure itself, and the role of queue backlogs as Lagrange multipliers
(see also  \cite{eryilmaz_qbsc_ton07} \cite{huangneely_dr_tac}).  However, this performance improvement 
is not for free: We must  \emph{drop a small fraction of packets} in order to dramatically
improve delay for the remaining ones.  We prove that as the $V$ parameter is
increased, the fraction of dropped packets quickly converges to zero, while
maintaining $O(1/V)$ close-to-optimal utilitiy and $O([\log(V)]^2)$ average backlog. 
This provides an analytical justification for experimental observations 
in \cite{scott_lifo_ipsn} that shows a related LIFO-Backpressure rule serves up to $98\%$ of the traffic with delay that is improved by 2 orders of magnitude.

LIFO-Backpressure was proposed in recent empirical work \cite{scott_lifo_ipsn}. The authors developed a practical implementation of backpressure routing and showed experimentally that applying LIFO queuing discipline drastically improves average packet delay, but did not provide theoretical guarantees. 
Another notable recent work providing an alternative delay solution is  \cite{athanasopoulouTON10}, which describes a novel backpressure-based per-packet randomized routing framework that runs atop the shadow queue structure of \cite{buimulticast} while minimizing hop count as explored in \cite{buisrikant_infocom09}.  Their techniques reduce delay drastically and eliminates the per-destination queue complexity, but does not provide $O([\log(V)]^2)$ average delay guarantees.



Our analysis of the delay performance of LIFO-Backpressure is based on the recent ``exponential attraction'' result developed in \cite{huangneely_dr_tac}. The proof  idea can be intuitively explained by Fig. \ref{fig:lifo-idea}, which depicts a simulated backlog process of a single queue system with unit packet size under Backpressure.  
\begin{figure}[cht]
\centering
\includegraphics[height=1.4in, width=2.1in]{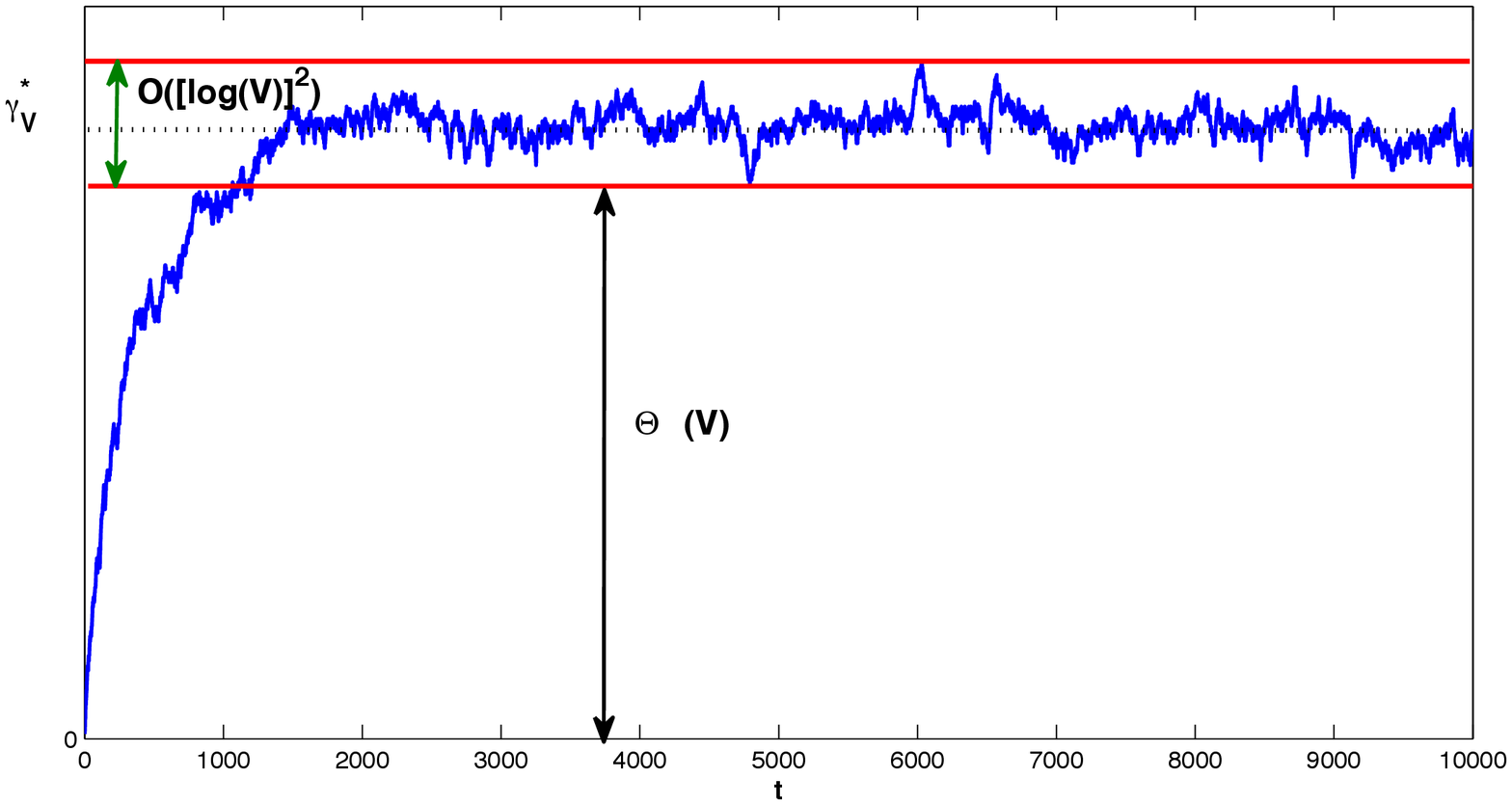}
\includegraphics[height=1.4in, width=1.2in]{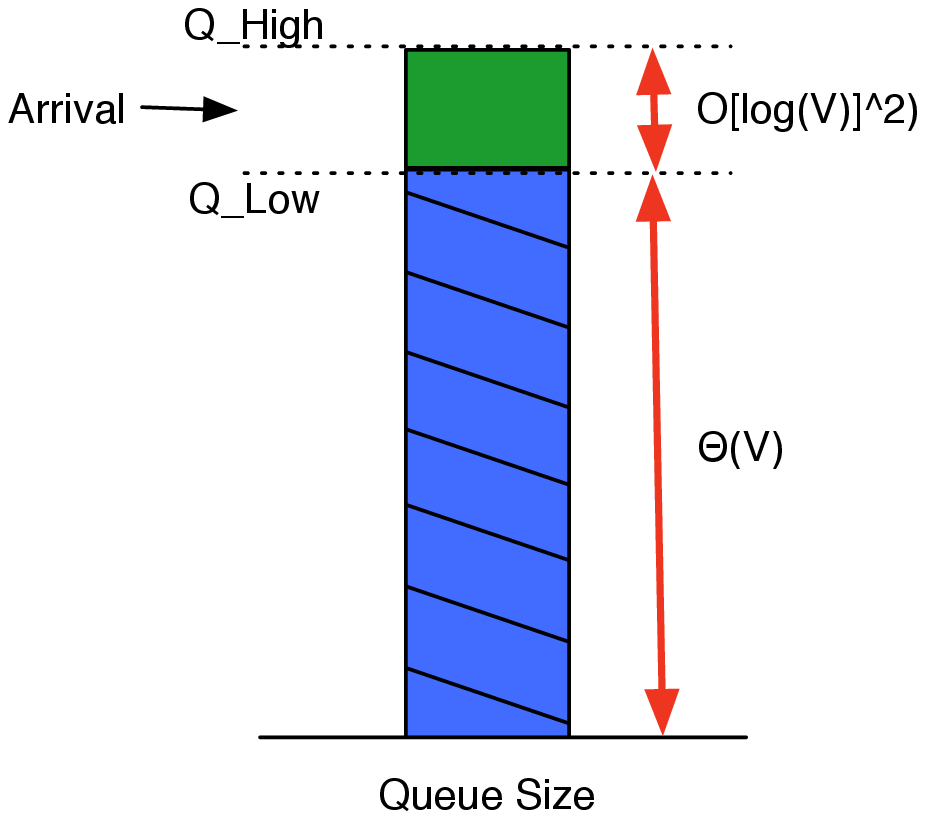}
\caption{The LIFO-Backpressure Idea}\label{fig:lifo-idea}
\end{figure}
The left figure demonstrates the  ``exponential attraction'' result in \cite{huangneely_dr_tac}, which states that queue sizes under Backpressure deviate from some fixed point with probability that decreases exponentially in the deviation distance. 
Hence the queue size will mostly fluctuate within the interval $[Q_{Low}, Q_{High}]$ which can be shown to be of $O([\log(V)]^2)$ size. This result holds under both FIFO and LIFO, as they result in the same queue process. 
Now suppose LIFO is used in this queue. Then from the right figure, we see that most of the packets will arrive at the queue when the queue size is between $Q_{Low}$ and $Q_{High}$, and  these new packets will always be placed on the top of the queue due to the LIFO discipline. Most packets thus enter and leave the queue when the queue size is between $Q_{Low}$ and $Q_{High}$. Therefore, these packets ``see'' a queue with average size no more than $Q_{High}-Q_{Low}=O([\log(V)]^2)$. 
Now let $\lambda$ be the packet arrival rate into the queue, and let 
$\tilde{\lambda}$ be the arrival rate of packets entering when the queue size is in $[Q_{Low}, Q_{High}]$ and that eventually depart.  
Because packets always occupy the same buffer slot under LIFO, we see that these packets can occupy at most $Q_{High}-Q_{Low}+\delta_{max}$ buffer slots, ranging from $Q_{Low}$ to $Q_{High}+\delta_{max}$, where $\delta_{max}=\Theta(1)$ is the maximum number of packets that can enter the queue at any time. 
We can now apply Little's Theorem  \cite{bertsekas_datanet} to the buffer slots in the interval $[Q_{Low}, Q_{High}+\delta_{max}]$, and we see that average delay for these packets that arrive when the queue size is in $[Q_{Low}, Q_{High}]$ satisfies: 
\begin{eqnarray}
D \leq \frac{Q_{High}-Q_{Low}+\delta_{max}}{\tilde{\lambda}} = \frac{O([\log(V)]^2)}{\tilde{\lambda}}. \label{eq:lifo-demo}
\end{eqnarray}

Finally, the exponential attraction result implies that $\lambda\approx\tilde{\lambda}$. Hence for almost all packets entering the queue, the average delay is $D=O([\log(V)]^2/\lambda)$. 


This paper is organized as follows. In Section \ref{section:notation}, we set up our notations. We then present our system model in Section \ref{section:model}. We provide an example of our network in Section \ref{section:example}. 
We review the Backpressure algorithm in Section \ref{section:qlareview}.  The delay performance of LIFO-Backpressure is presented in Section \ref{section:LIFO-Backpressure-analysis}. Simulation results are presented in Section \ref{section:simulation}. We then also present experimental testbed results in Section \ref{section:experiment}. Finally, we comment on optimizing a function of time averages in Section \ref{section:opt-timeavg}.

$\vspace{-.3in}$
\section{Notations}\label{section:notation}
Here we first set up the notations used in this paper:
$\mathbb{R}$ represents  the set of real numbers. 
$\mathbb{R}_+$ (or $\mathbb{R}_-$) denotes the set of nonnegative (or non-positive) real numbers. 
$\mathbb{R}^n$ (or $\mathbb{R}^n_+$) is  the set of $n$ dimensional \emph{column} vectors, with each element being in $\mathbb{R}$ (or $\mathbb{R}_+$). 
\textbf{bold} symbols $\bv{a}$ and $\bv{a}^T$ represent \emph{column} vector and its transpose. 
$\bv{a}\succeq\bv{b}$ means vector $\bv{a}$ is entrywise no less than vector $\bv{b}$. 
$||\bv{a}-\bv{b}||$ is the Euclidean distance of $\bv{a}$ and $\bv{b}$. 
$\bv{0}$ and $\bv{1}$ denote column vector with all elements being $0$ and $1$. 
$[a]^+=\max[a, 0]$ and $\log(\cdot)$ is the natural log.



$\vspace{-.2in}$
\section{System Model}\label{section:model}
In this section, we specify the general network model we use. We consider a network controller that operates a network with the goal of minimizing the time average cost, subject to the queue stability constraint. The network is assumed to operate in slotted time, i.e., $t\in\{0,1,2,...\}$. We assume there are $r\geq1$ queues in the network. 

$\vspace{-.22in}$
\subsection{Network State}
In every slot $t$, we use $S(t)$ to denote the current network state, which indicates the current network parameters, such as a vector of channel conditions for each link, or a collection of other relevant information about the current network channels and arrivals. 
We assume that $S(t)$ evolves according a finite state irreducible and aperiodic Markov chain, with a total of $M$ different random network states denoted as $\script{S} = \{s_1, s_2, \ldots, s_M\}$.  Let $\pi_{s_i}$ denote the steady state probability of being in state $s_i$. It is easy to see in this case that $\pi_{s_i}>0$ for all $s_i$. The network controller can observe $S(t)$ at the beginning of every slot $t$, but the $\pi_{s_i}$ and transition probabilities are not necessarily known.

$\vspace{-.22in}$
\subsection{The Cost, Traffic, and Service}\label{subsection:costtrafficservice}
At each time $t$, after observing $S(t)=s_i$, the controller chooses an action $x(t)$ from a set $\script{X}^{(s_i)}$, i.e., $x(t)= x^{(s_i)}$ for some $x^{(s_i)}\in\script{X}^{(s_i)}$. The set $\script{X}^{(s_i)}$ is called the feasible action set for network state $s_i$ and is assumed to be time-invariant and compact for all $s_i\in\script{S}$.  The cost, traffic, and service generated by the chosen action $x(t)=x^{(s_i)}$ are as follows:
\begin{enumerate}
\item[(a)] The chosen action has an associated cost given by the cost function $f(t)=f(s_i, x^{(s_i)}): \script{X}^{(s_i)}\mapsto \mathbb{R}_+$ (or $\script{X}^{(s_i)}\mapsto\mathbb{R}_-$ in reward maximization problems);

\item[(b)] The amount of traffic generated by the action to queue $j$ is determined by the traffic function $A_j(t)=A_{j}(s_i, x^{(s_i)}): \script{X}^{(s_i)}\mapsto \mathbb{R}_{+}$, in units of packets; 

\item[(c)] The amount of service allocated to queue $j$ is given by the rate function $\mu_j(t)=\mu_{j}(s_i, x^{(s_i)}): \script{X}^{(s_i)}\mapsto \mathbb{R}_{+}$, in units of packets;

 \end{enumerate}
Note that $A_j(t)$ includes both the exogenous arrivals from outside the network to queue $j$, and the endogenous arrivals from other queues, i.e., the transmitted packets from other queues, to queue $j$. We assume the functions $f(s_i, \cdot)$, $\mu_{j}(s_i, \cdot)$ and $A_{j}(s_i, \cdot)$ are continuous, time-invariant, their magnitudes are uniformly upper bounded by some constant $\delta_{max}\in(0,\infty)$ for all $s_i$, $j$, and they are known to the network operator. We also assume that there exists a set of actions $\{x^{(s_i)k}\}_{i=1,..., M}^{k=1,2, ..., \infty}$ with $x^{(s_i)k}\in\script{X}^{(s_i)}$ and some variables $\vartheta^{(s_i)}_k\geq0$ for all $s_i$ and $k$ with $\sum_k\vartheta^{(s_i)}_k=1$ for all $s_i$, such that 
\begin{eqnarray}
\sum_{s_i}\pi_{s_i}\big\{\sum_k\vartheta^{(s_i)}_k[A_{j}(s_i, x^{(s_i)k})-\mu_{j}(s_i, x^{(s_i)k})]\big\}\leq -\eta, \label{eq:slackness}
\end{eqnarray}
for some $\eta>0$ for all $j$. That is, the stability 
constraints are feasible with $\eta$-slackness. Thus, there exists a stationary randomized policy that stabilizes all queues (where $\vartheta^{(s_i)}_k$ represents the probability of choosing action $x^{(s_i)k}$ when $S(t)=s_i$) \cite{neelynowbook}.  


$\vspace{-.2in}$
\subsection{Queueing, Average Cost, and the Stochastic Problem}\label{section:queuenotation}
Let $\bv{q}(t)=(q_1(t), ..., q_r(t))^T\in\mathbb{R}^r_{+}$, $t=0, 1, 2, ...$ be the queue backlog vector  process of the network, in units of packets. We assume the following queueing dynamics: 
\begin{eqnarray}
q_j(t+1)=\max\big[q_j(t)-\mu_j(t), 0\big]+A_j(t)\quad\forall j,\label{eq:queuedynamic}
\end{eqnarray}
and $\bv{q}(0)=\bv{0}$. By using (\ref{eq:queuedynamic}), we assume that when a queue does not have enough packets to send, null packets are transmitted.  In this paper, we adopt the following notion of queue stability:
\begin{eqnarray}
\expect{\sum_{j=1}^rq_j}\triangleq
\limsup_{t\rightarrow\infty}\frac{1}{t}\sum_{\tau=0}^{t-1}\sum_{j=1}^{r}\expect{q_j(\tau)}<\infty.\label{eq:queuestable}
\end{eqnarray}
We also use $f^{\Pi}_{av}$ to denote the time average cost induced by an action-choosing policy $\Pi$, defined as:
\begin{eqnarray}
f^{\Pi}_{av}\triangleq
\limsup_{t\rightarrow\infty}\frac{1}{t}\sum_{\tau=0}^{t-1}\expect{f^{\Pi}(\tau)},\label{eq:timeavcost}
\end{eqnarray}
where $f_{av}^{\Pi}(\tau)$ is the cost incurred at time $\tau$ by policy $\Pi$. We call an action-choosing  policy \emph{feasible} if at every time slot $t$ it only chooses actions from the feasible action set $\script{X}^{(S(t))}$.  We then call a feasible action-choosing  policy under which (\ref{eq:queuestable}) holds a \emph{stable} policy, and use $f_{av}^*$ to denote the optimal time average cost over all stable policies. 
In every slot, the network controller observes the current network state and chooses a control action, with the goal of minimizing the time average cost subject to network stability. This goal can be mathematically stated as:
\textbf{(P1)}\,\,\, $\bv{\min: \,  f^{\Pi}_{av}, \,\, s.t.\,  (\ref{eq:queuestable})}$. 
In the following, we will refer to \textbf{(P1)} as \emph{the stochastic problem}. 

Note that in some network optimization problems, e.g., \cite{neelyfairness}, the objective of the network controller is to optimize a function of a time average metric. In this case, we see that the Backpressure algorithm and the deterministic problem presented in the next section can similarly be constructed, but will be slightly different. 
We will discuss these problems in Section \ref{section:opt-timeavg}. 

\vspace{-.1in}
\section{An example of our model}\label{section:example}
Here we provide an example to illustrate our model. Consider the $2$-queue network in Fig. \ref{fig:net-example}. In every slot, the network operator decides whether or not to allocate one unit of power to serve packets at each queue, so as to support all arriving traffic, i.e., maintain queue stability, with minimum energy expenditure. 
We assume the network state $S(t)$, which is the quadruple $(R_1(t), R_2(t), CH_1(t), CH_2(t))$, evolves according to the finite state Markov chain with three states $s_1=(1, 1, G, B), s_2=(1, 1, G, G)$, and $s_3=(0, 0, B, G)$. Here $R_i(t)$ denotes the number of exogenous packet arrivals to queue $i$ at time $t$, and $CH_i(t)$ is the state of channel $i$. $R_i(t)=x$ implies that there are $x$ number of packets arriving at queue $i$ at time $t$. $CH_i(t)=G/B$ means that channel $i$ has a ``Good'' or ``Bad'' state. When a link's  channel state is ``Good'', one unit of power can serve $2$ packets over the link, otherwise it can only serve one. We assume power can be allocated to both channels without affecting each other. 



\begin{figure}[cht]
\centering
\includegraphics[height=0.7in, width=2.5in]{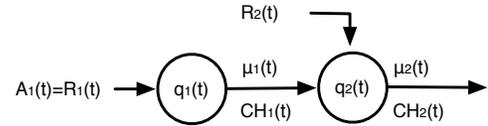}
\caption{A two queue tandem example.}\label{fig:net-example}
\end{figure}


In this case, we see that there are three possible network states. At each state $s_i$, the action $x^{(s_i)}$ is a pair $(x_1, x_2)$, with $x_i$ being the amount of energy spent at queue $i$, and $(x_1, x_2)\in\script{X}^{(s_i)}=\{0/1,0/1\}$. The cost function is  $f(s_i, x^{(s_i)})=x_1+x_2$, for all $s_i$. The network states, the traffic functions,  and the service rate functions are summarized in Fig. \ref{fig:net-function}. Note here $A_1(t)=R_1(t)$ is part of $S(t)$ and is independent of $x^{(s_i)}$; while $A_2(t)=\mu_1(t)+R_2(t)$ hence depends on $x^{(s_i)}$. Also note that $A_2(t)$ equals $\mu_1(t)+R_2(t)$ instead of $\min[\mu_1(t), q_1(t)]+R_2(t)$ due to our idle fill assumption in Section \ref{section:queuenotation}. 

\begin{figure}[cht]
\centering
\includegraphics[height=0.7in, width=3.5in]{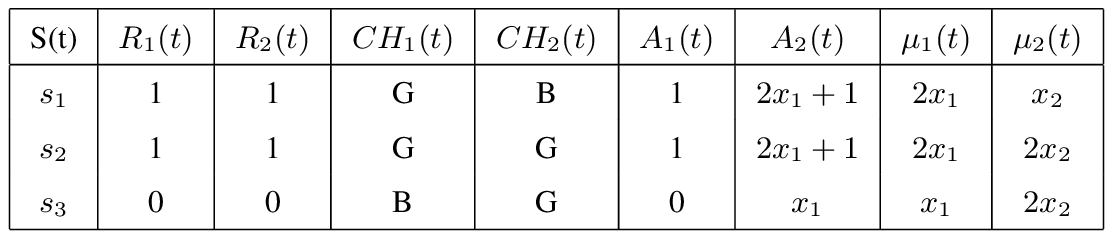}
\caption{The traffic and service functions under different states.}\label{fig:net-function}
\end{figure}

\vspace{-.1in}
\section{Backpressure and the Deterministic Problem}\label{section:qlareview}
In this section, we first review the Backpressure algorithm \cite{neelynowbook} for solving the stochastic problem. Then we define the \emph{deterministic problem} and its dual. 
We first recall the Backpressure algorithm for utility optimization problems \cite{neelynowbook}. 

\underline{\emph{Backpressure:}} At every time slot $t$, observe the current network state $S(t)$ and the backlog $\bv{q}(t)$. If $S(t)=s_i$, choose $x^{(s_i)}\in\script{X}^{(s_i)}$ that solves the following: 
\vspace{-.15in}
\begin{eqnarray}
\hspace{-.3in}\max: && -Vf(s_i, x)+\sum_{j=1}^{r}q_j(t)\big[\mu_j(s_i, x)-A_j(s_i, x)\big]\label{eq:QLAeq}\\
s.t. && x\in\script{X}^{(s_i)}.\nonumber
\end{eqnarray}
Depending on the problem structure, (\ref{eq:QLAeq}) can usually be decomposed into separate parts that are easier to solve, e.g., \cite{huangneelypricing}, \cite{rahulneelycognitive}. 
Also, when the network state process $S(t)$ is i.i.d., it has been shown in \cite{neelynowbook} that, 
\begin{eqnarray}
f_{av}^{BP}=f^*_{av}+O(1/V),\quad \overline{q}^{BP}=O(V),\label{eq:qla_performance}
\end{eqnarray}
where $f_{av}^{BP}$ and $\overline{q}^{BP}$ are the expected average cost and the expected average network backlog size under Backpressure, respectively. 
When $S(t)$ is Markovian, \cite{huangneelypricing} and \cite{rahulneelycognitive} show that Backpressure achieves an $[O(\log(V)/V), O(V)]$ utility-delay tradeoff if the queue sizes are deterministically upper bounded by $\Theta(V)$ for all time. Without this deterministic backlog bound, it has recently been shown that Backpressure achieves an $[O(\epsilon+\frac{T_{\epsilon}}{V}), O(V)]$ tradeoff under Markovian $S(t)$, with $\epsilon$ and $T_{\epsilon}$ representing the proximity to the optimal utility value and the ``convergence time'' of the Backpressure algorithm to that proximity \cite{neely_queuestability10}. However, there has not been any formal proof that shows the exact $[O(1/V), O(V)]$ utility-delay tradeoff of Backpressure under a Markovian $S(t)$. 



We also recall  \emph{the deterministic problem} defined in \cite{huangneely_dr_tac}:
\begin{eqnarray}
\min:&&\script{F}(\bv{x})\triangleq V\sum_{s_i}\pi_{s_i}f(s_i, x^{(s_i)})\label{eq:primal}\\
s.t.&&\script{A}_j(\bv{x})\triangleq\sum_{s_i}\pi_{s_i}A_j(s_i, x^{(s_i)})\nonumber\\
&&\qquad\qquad\quad\leq \script{B}_j(\bv{x})\triangleq\sum_{s_i}\pi_{s_i}\mu_j(s_i, x^{(s_i)}),\,\,\forall\, j,\nonumber\\
&& x^{(s_i)}\in \script{X}^{(s_i)}\quad \forall\, i=1, 2, ..., M,\nonumber
\end{eqnarray}
where $\pi_{s_i}$ corresponds to the steady state probability of $S(t)=s_i$ and $\bv{x}=(x^{(s_1)}, ..., x^{(s_M)})^T$. The dual problem of (\ref{eq:primal}) can be obtained as follows:
\begin{eqnarray}
\max:\,\,\, g(\bv{\gamma}),\quad s.t.\,\,\, \bv{\gamma}\succeq\bv{0},\label{eq:dualproblem}
\end{eqnarray}
where $g(\bv{\gamma})$ is called the dual function and is defined as:
\begin{eqnarray}
\hspace{-.3in}&&g(\bv{\gamma})=\inf_{x^{(s_i)}\in \script{X}^{(s_i)}}\sum_{s_i}\pi_{s_i}\bigg\{Vf(s_i, x^{(s_i)})\label{eq:dual_separable}\\
\hspace{-.3in}&&\qquad\qquad\qquad\qquad+\sum_j\gamma_j\big[A_j(s_i, x^{(s_i)})- \mu_j(s_i, x^{(s_i)})\big]\bigg\}.\nonumber
\end{eqnarray}

Here $\bv{\gamma}=(\gamma_1, ..., \gamma_r)^T$ is the \emph{Lagrange multiplier} of (\ref{eq:primal}). It is well known that $g(\bv{\gamma})$ in (\ref{eq:dual_separable}) is concave in the vector $\bv{\gamma}$, and hence the problem (\ref{eq:dualproblem}) can usually be solved efficiently, particularly when cost functions and rate functions are separable over different network components. Below, we use $\bv{\gamma}_V^*=(\gamma^*_{V1}, \gamma^*_{V2}, ..., \gamma^*_{Vr})^T$ to denote an optimal solution of the problem (\ref{eq:dualproblem}) with the corresponding $V$.

\vspace{-.1in}
\section{Performance of LIFO Backpressure}\label{section:LIFO-Backpressure-analysis}
In this section, we analyze the performance of Backpressure with the LIFO queueing discipline (called LIFO-Backpressure). The idea of using LIFO under Backpressure is first proposed in \cite{scott_lifo_ipsn}, although they did not provide any theoretical performance guarantee. We will show, under some mild conditions (to be stated in Theorem \ref{thm:prob_multi_con}), that under LIFO-Backpressure, the time average delay for almost all packets entering the network is $O([\log(V)]^2)$ when the utility is pushed to within $O(1/V)$ of the optimal value. Note that the implementation complexity of LIFO-Backpressure is the same as the original Backpressure, and LIFO-Backpressure only requires the knowledge of the instantaneous network condition. This is a remarkable feature that distinguishes it from the previous algorithms  achieving similar poly-logarithmic tradeoffs in the i.i.d. case, e.g., \cite{neelysuperfast} \cite{neelyenergydelay} \cite{huangneely_dr_tac}, which all require knowledge of some implicit network parameters other than the instant network state. 
Below,  we first provide a simple example to demonstrate the need for careful treatment of the usage of LIFO in Backpressure algorithms, and then present a modified Little's theorem that will be used for our proof.

\vspace{-.1in}
\subsection{A simple example on the LIFO delay}
Consider a slotted system where two packets arrive at time $0$, and one packet periodically arrives every slot thereafter
(at times $1, 2, 3,\ldots$). The system is initially empty and can serve exactly one packet per slot.  The arrival rate $\lambda$
is clearly $1$ 
packet/slot (so that $\lambda = 1$).  Further, under either FIFO or LIFO service, there are always 2 packets in the system, 
so $\overline{Q} = 2$. 

Under FIFO service, 
the first packet has a delay of $1$ and all packets thereafter have a delay of $2$:
\[ W_1^{FIFO} = 1 \: , \: W_i^{FIFO} = 2 \: \: \forall i \in \{2, 3, 4, \ldots\}, \]
where $W_i^{FIFO}$ is the delay of the $i^{th}$ packet under FIFO ($W_i^{LIFO}$ is similarly defined for LIFO). 
We thus have: 
\[ \overline{W}^{FIFO} \defequiv \lim_{K\rightarrow\infty} \frac{1}{K}\sum_{i=1}^KW_i^{FIFO} = 2. \]
Thus,  $\lambda \overline{W}^{FIFO} = 1\times 2 = 2$, $\overline{Q} = 2$, and so $\lambda\overline{W}^{FIFO} = \overline{Q}$
indeed holds. 

Now consider the same system under LIFO service.  We still have $\lambda =1$, $\overline{Q}=2$. 
However, in this case the first packet never departs, while all other packets have a
delay equal to $1$ slot: 
\[ W_1^{LIFO} = \infty \: , \: W_i^{LIFO} = 1 \: \: \forall i \in \{2, 3, 4, \ldots\}. \]
Thus, for all integers $K>0$: 
\[ \frac{1}{K}\sum_{i=1}^K W_i^{LIFO} = \infty. \]
and so $\overline{W}^{LIFO} = \infty$.  Clearly $\lambda \overline{W}^{LIFO} \neq \overline{Q}$. 
 On the other hand, if we 
ignore the one packet with infinite delay, we note that all other packets get a delay of 1 (exactly half the delay in the FIFO
system).  Thus, in this example, LIFO service  significantly improves delay for all but the first packet. 

For the above LIFO example, it is interesting to note that
if we define $\tilde{Q}$ and $\tilde{W}$ as the average backlog and delay
\emph{associated only with those packets that eventually depart}, then we have $\tilde{Q} = 1$, $\tilde{W}=1$, and the
equation $\lambda \tilde{W} = \tilde{Q}$ indeed holds.  This motivates the theorem in the next subsection,   which 
considers a time average only over those packets that eventually depart. 

\vspace{-.1in}
\subsection{A Modified Little's Theorem for LIFO systems} \label{section:LIFO-little}
We now present the modified Little's theorem. 
Let $\script{B}$ represent a finite set of buffer locations for a LIFO queueing system. 
Let $N(t)$ be the number of arrivals that use a buffer
location within set $\script{B}$ up to time $t$.  Let $D(t)$ be the number of departures from a buffer location within the set 
$\script{B}$ up to time $t$.  Let $W_i$ be the delay of the $i$th job to depart from the set $\script{B}$.  Define $\overline{W}$ as the
$\limsup$ average delay \emph{considering only those jobs that depart}: 
\[ \overline{W} \defequiv \limsup_{t\rightarrow\infty} \frac{1}{D(t)}\sum_{i=1}^{D(t)} W_i. \]
We then have the following theorem: 
\begin{thm} \label{theorem:little1} 
Suppose there is a constant $\lambda_{min}>0$ such that with probability 1: 
\[ \liminf_{t\rightarrow\infty} \frac{N(t)}{t} \geq \lambda_{min}, \]
Further suppose that $\lim_{t\rightarrow\infty} D(t) = \infty$ with probability 1 (so the number of departures is infinite). 
Then the average delay $\overline{W}$ satisfies: 
\[ \overline{W} \defequiv \limsup_{t\rightarrow\infty} \frac{1}{D(t)}\sum_{i=1}^{D(t)} W_i \leq |\script{B}|/\lambda_{min}, \]
where $|\script{B}|$ is the size of the finite set $\script{B}$. 
\end{thm} 
\begin{proof}
See Appendix A. 
\end{proof}

\subsection{LIFO-Backpressure Proof}
We now provide the analysis of LIFO-Backpressure. 
To prove our result, we first have the following theorem, which is the first to show that Backpressure (with either FIFO or LIFO) achieves the exact $[O(1/V), O(V)]$ utility-delay tradeoff under a Markovian network state process. It generalizes the $[O(1/V), O(V)]$ performance result of Backpressure in the i.i.d. case in \cite{neelynowbook}. 
\begin{thm}\label{theorem:qla-markovian}
Suppose $S(t)$ is a finite state irreducible and aperiodic Markov chain\footnote{In \cite{huangneely_qlamarkovian}, the theorem is proven under more general Markovian $S(t)$ processes that include the $S(t)$ process assumed here.}  and condition (\ref{eq:slackness}) holds, Backpressure (with either FIFO or LIFO) achieves the following:
\begin{eqnarray}
f_{av}^{BP}=f^*_{av} + O(1/V), \,\, \overline{q}^{BP} = O(V), 
\end{eqnarray}
where $f_{av}^{BP}$ and $\overline{q}^{BP}$ are the expected time average cost and  backlog under Backpressure. 
\end{thm}
\begin{proof}
See \cite{huangneely_qlamarkovian}. 
\end{proof}

Theorem \ref{theorem:qla-markovian} thus shows that LIFO-Backpressure guarantees an average backlog of $O(V)$ when pushing the utility to within $O(1/V)$ of the optimal value. 
We now consider the delay performance of LIFO-Backpressure. For our analysis, we need the following theorem (which is Theorem 1 in \cite{huangneely_dr_tac}).
\begin{thm}\label{thm:prob_multi_con}
Suppose that $\bv{\gamma}^*_V$ is unique,  that the slackness condition (\ref{eq:slackness}) holds, and that  the dual function $g(\bv{\gamma})$ satisfies:
\begin{eqnarray}
g(\bv{\gamma}^*_V)\geq g(\bv{\gamma})+L||\bv{\gamma}^*_V-\bv{\gamma}||\quad\forall\,\,\bv{\gamma}\succeq\bv{0}\label{eq:dualpolyhedralcond},
\end{eqnarray}
for some constant $L>0$ independent of $V$. Then under Backpressure with FIFO or LIFO, there exist constants $D, K, c^*=\Theta(1)$, i.e., all independent of $V$, such that for any $m\in \mathbb{R}_+$, 
\begin{eqnarray}
\script{P}^{(r)}(D, Km)&\leq& c^*e^{-m},\label{eq:prob_pmr_special}
\end{eqnarray}
where $\script{P}^{(r)}(D, Km)$ is defined:
\begin{eqnarray}
\hspace{-.3in}&&\script{P}^{(r)}(D, Km)\label{eq:pmr_def}\\
\hspace{-.3in}&&\qquad\quad\triangleq\limsup_{t\rightarrow\infty}\frac{1}{t}\sum_{\tau=0}^{t-1}\prob\{\exists\, j, |q_j(\tau)-\gamma_{Vj}^*|>D+Km\}.\nonumber
\end{eqnarray}
\end{thm}
\begin{proof}
See \cite{huangneely_dr_tac}. 
\end{proof}

Note that if a steady state distribution exists for $\bv{q}(t)$, e.g., when all queue sizes are integers, then $\script{P}^{(r)}(D, Km)$ is indeed the steady state probability that there exists a queue $j$ whose queue value deviates from $\gamma^*_{Vj}$ by more than $D+Km$ distance. In this case, Theorem \ref{thm:prob_multi_con} states that $q_j(t)$ deviates from $\gamma_{Vj}^*$ by $\Theta(\log(V))$ distance with probability $O(1/V)$. Hence when $V$ is large, $q_j(t)$ will mostly be within $O(\log(V))$ distance from $\gamma^*_{Vj}$. 
Also note that the conditions of Theorem  \ref{thm:prob_multi_con} are not very restrictive. The condition (\ref{eq:dualpolyhedralcond}) can usually be satisfied in practice when the action space is  finite, in which case the dual function $g(\bv{\gamma})$ is polyhedral (see  \cite{huangneely_dr_tac} for more discussions). The uniqueness of $\bv{\gamma}_V^*$ can usually be satisfied in many network utility optimization problems, e.g., \cite{eryilmaz_qbsc_ton07}. 


We now present the main result of this paper with respect to the delay performance of LIFO-Backpressure. 
Below, the notion ``average arrival rate'' is defined as follows:  Let $A_j(t)$ be the number of packets
entering queue $j$ at time $t$.  Then the time average arrival rate  of these packets is defined (assuming
it exists): 
$\lambda_j = \lim_{t\rightarrow\infty}\frac{1}{t} \sum_{\tau=0}^{t-1}A_j(\tau)$. 
For the theorem, we assume that time averages under Backpressure exist with probability 1. This is a reasonable assumption, and holds whenever the resulting discrete time Markov chain for the queue vector $\bv{q}(t)$ under backpressure is countably infinite and irreducible.  Note that the 
state space is indeed countably infinite if we assume packets take integer units.  If the system is also irreducible then the finite average backlog 
result of Theorem \ref{theorem:qla-markovian} implies that all states are positive recurrent. 

Let $D, K, c^*$ be constants as defined in Theorem \ref{thm:prob_multi_con}, and 
recall that these are $\Theta(1)$ (independent of $V$). 
Assume $V\geq1$, and define $Q_{j, High}$ and $Q_{j, Low}$ as: 
\begin{eqnarray*}
Q_{j, High} &\defequiv&  \gamma_{V j}^* + D + K[\log(V)]^2, \\
Q_{j, Low} &\defequiv& \max[\gamma_{V j}^* - D - K[\log(V)]^2, 0]. 
\end{eqnarray*}
Define the interval 
$\script{B}_j \defequiv [Q_{j, High}, Q_{j, Low}]$.  The following theorem considers
the rate and delay of packets that enter when $q_j(t) \in \script{B}_j$ and that eventually depart. 


\begin{thm} \label{thm:qlalifo_delay}
Suppose that $V\geq1$, that 
$\bv{\gamma}_V^*$ is unique, that the slackness assumption 
(\ref{eq:slackness}) holds, and that the dual function $g(\bv{\gamma})$ satisfies: 
\begin{eqnarray}
g(\bv{\gamma}^*_V)\geq g(\bv{\gamma})+L||\bv{\gamma}^*_V-\bv{\gamma}||\quad\forall\,\,\bv{\gamma}\succeq\bv{0}\label{eq:dualpolyhedralcond2},
\end{eqnarray}
for some constant $L>0$ independent of $V$.  Define $D, K, c^*$ as in Theorem  \ref{thm:prob_multi_con}, and define $\script{B}_j$ as above. 
Then for any queue $j$ with a time average input rate $\lambda_j>0$, 
we have under LIFO-Backpressure that: 

(a) The rate $\tilde{\lambda}_j$ of packets that both arrive to queue $j$ when $q_j(t) \in\script{B}_j$ and that eventually 
depart the queue satisfies: 
\begin{eqnarray}
\lambda_j \geq \tilde{\lambda}_j \geq \left[\lambda_j - \frac{\delta_{max}c^*}{V^{\log(V)}}\right]^+. \label{eq:rate-major}
\end{eqnarray}

(b) The average delay of these packets is at most $W_{bound}$, where: 
\[ W_{bound} \defequiv [2D + 2K[\log(V)]^2+\delta_{max}]/\tilde{\lambda}_j. \]
\end{thm} 

This theorem says that the delay of packets that enter when $q_j(t) \in \script{B}_j$ and that eventually 
depart is at most $O([\log(V)]^2)$.  Further, by (\ref{eq:rate-major}), when $V$ is large, these packets represent the overwhelming majority,
 in that the rate of packets not in this set is at most $O(1/V^{\log(V)})$. 

\begin{proof} (Theorem \ref{thm:qlalifo_delay})   Theorem \ref{theorem:qla-markovian} shows that average queue backlog is finite.  Thus, there can be at most a finite
number of packets that enter the queue and never depart, so the rate of packets arriving that never depart must be $0$.  
It follows that $\tilde{\lambda}_j$ is equal to the rate at 
which packets arrive when $q_j(t) \in \script{B}_j$. 
Define the indicator function $1_j(t)$ to be $1$ if 
$q_j(t) \notin \script{B}_j$, and $0$ else.  Define $\tilde{\lambda}_j^c \defequiv \lambda_j - \tilde{\lambda}_j$. 
Then with probability 1 we get: \footnote{The 
time average expectation is the same as the pure time average by the Lebesgue Dominated Convergence
Theorem, because we assume the pure time average exists with
probability 1, and that $0 \leq A_j(t) \leq \delta_{max} \: \forall t$.} 
\begin{eqnarray*}
 \tilde{\lambda}_j^c &=& \lim_{t\rightarrow\infty} \frac{1}{t}\sum_{\tau=0}^{t-1} A_j(\tau)1_j(\tau) \\
 &=& \lim_{t\rightarrow\infty} \frac{1}{t}\sum_{\tau=0}^{t-1}\expect{A_j(\tau)1_j(\tau)}.  
\end{eqnarray*}
Then using the fact that $A_j(t)\leq\delta_{max}$ for all $j, t$: 
\begin{eqnarray*}
 \expect{A_j(t)1_j(t)} &=& \expect{A_j(t)|  q_j(t) \notin \script{B}_j}\prob\{q_j(t)\notin\script{B}_j) \\
 &\leq& \delta_{max} \prob(q_j(t)\notin[Q_{j,Low}, Q_{j,High}]). 
\end{eqnarray*}
Therefore: 
\begin{eqnarray*}
\tilde{\lambda}_j^c &\leq& \delta_{max}\lim_{t\rightarrow\infty} \frac{1}{t}\sum_{\tau=0}^{t-1}\prob(q_j(\tau) \notin [Q_{j,Low}, Q_{j,High}]) \\
&\leq& \delta_{max} \lim_{t\rightarrow\infty} \frac{1}{t}\sum_{\tau=0}^{t-1}\prob(|q_j(\tau) - \gamma_{V, j}^*| > D + Km), 
\end{eqnarray*}
where we define $m \defequiv [\log(V)]^2$, and note that $m\geq0$ because $V\geq1$. From Theorem \ref{thm:prob_multi_con} we thus have: 
\begin{eqnarray}
0 \leq \tilde{\lambda}_j^c \leq \delta_{max} c^* e^{-m}  = \frac{\delta_{max}c^*}{V^{\log(V)}}.
\end{eqnarray}
This completes the proof of part (a).  
Now define $\tilde{\script{B}_j}= \defequiv [Q_{j, High}+\delta_{max}, Q_{j, Low}]$. Since $\script{B}_j\subset\tilde{\script{B}_j}$, we see that the rate of the packets that enter $\tilde{\script{B}_j}$ is at least $\tilde{\lambda}_j$. 
Part (b) then follows from Theorem \ref{theorem:little1} and the facts that queue $j$ is stable and  that  $|\tilde{\script{B}_j}|\leq2D + 2K[\log(V)]^2+\delta_{max}$. 
\end{proof}

Note that if $\lambda_j=\Theta(1)$, we see from Theorem \ref{thm:qlalifo_delay} that, under LIFO-Backpressure, the time average delay for almost all packets   going through queue $j$ is only $O([\log(V)]^2)$. Applying this argument to all network queues with $\Theta(1)$ input rates, we see that all but a tiny fraction of the traffic entering the network only experiences a delay of $O([\log(V)]^2)$.  This contrasts with the delay performance result of the usual Backpressure with FIFO, which states that the time average delay will be $\Theta(V)$ for all packets \cite{huangneely_dr_tac}.
Also note that under LIFO-Backpressure, some packets may stay in the queue for very long time. This problem can be compensated by introducing certain coding techniques, e.g., fountain codes \cite{mitzenmacher-fountain}, into the LIFO-Backpressure algorithm.



%



\section{Simulation}\label{section:simulation}
In this section, we provide simulation results of the LIFO-Backpressure algorithm. 
We consider the network shown in Fig. \ref{fig:LIFO_net}, where we try to support a  flow sourced by Node $1$ destined for Node $7$ with minimum energy consumption. 

\begin{figure}[cht]
\centering
\includegraphics[height=1.1in, width=2.8in]{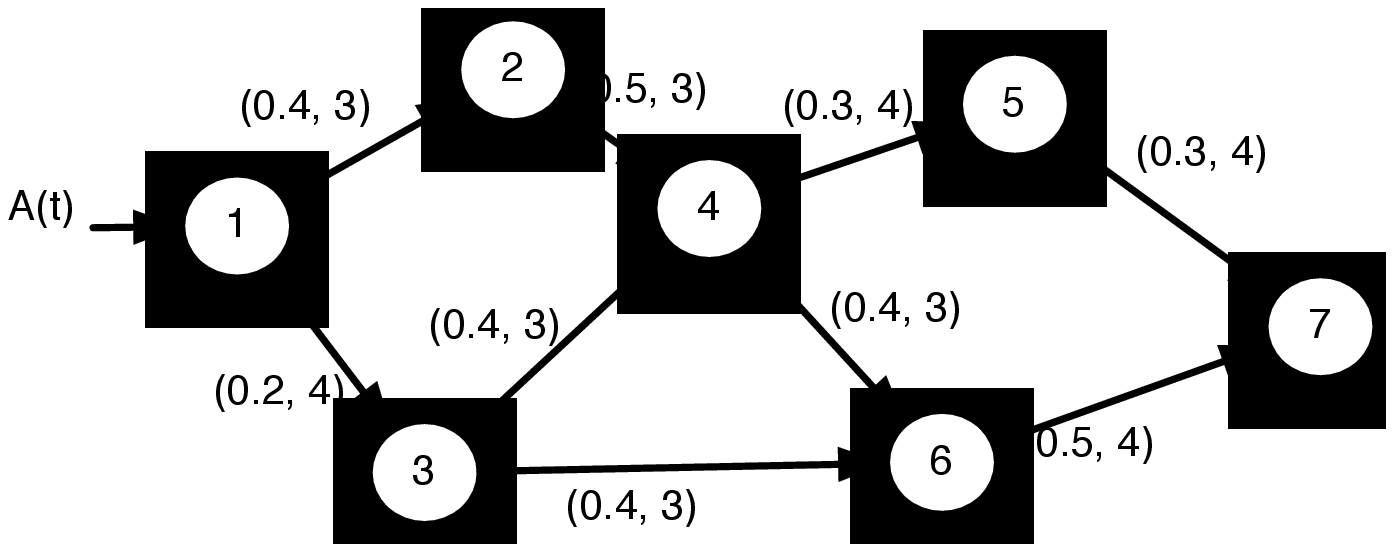}
\caption{A multihop network. $(a, b)$ represents the $HIGH$ probability $a$ and the rate $b$ obtained with one unit of power when $HIGH$.}\label{fig:LIFO_net}
\end{figure}

We  assume that $A(t)$ evolves according to the 2-state Markov chain in Fig. \ref{fig:markov-chain}. When the state is $HIGH$, $A(t)=3$, else $A(t)=0$. 
We assume that the condition of each link can either be $HIGH$ or $LOW$ at a time. 
All the links except link $(2, 4)$ and link $(6, 7)$ are assumed to be i.i.d. every time slot, whereas the conditions of link $(2, 4)$ and link $(6, 7)$ are assumed to be evolving according to independent 2-state Markov chains in Fig. \ref{fig:markov-chain}.
Each link's $HIGH$ probability and unit power rate at the $HIGH$ state is shown in Fig. \ref{fig:LIFO_net}. The unit power rates of the links at the $LOW$ state are all  assumed to be $1$. 
We assume that the link states are all independent and there is no interference. However, each node can only spend one unit of power per slot to transmit over one outgoing link, although it can simultaneously receive from multiple incoming links. The goal is to minimize the time average power while maintaining network stability.

\begin{figure}[cht]
\centering
\includegraphics[height=0.6in, width=1.6in]{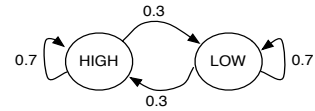}
\caption{The two state Markov chain with the transition probabilities.}\label{fig:markov-chain}
\end{figure}




We simulate Backpressure with both LIFO and FIFO for $10^6$ slots with $V\in\{20, 50, 100, 200, 500\}$. It can be verified that the backlog vector converges to a unique attractor as $V$ increases in this case.  The left two plots in Fig. \ref{fig:LIFO-power-backlog} show the average power consumption and the average backlog under LIFO-Backpressure. It can be observed that the average power quickly converges to the optimal value and that the average backlog grows linearly in $V$. The right plot of Fig. \ref{fig:LIFO-power-backlog} shows the percentage of time when there exists a $q_j$ whose value deviates from $\gamma_{Vj}^*$ by more than $2[\log(V)]^2$. As we can see, this percentage 
is always very small, i.e., between $0.002$ and $0.013$, showing a good match between the theory and the simulation results. 

\begin{figure}[cht]
\centering
\includegraphics[height=1.6in, width=3.3in]{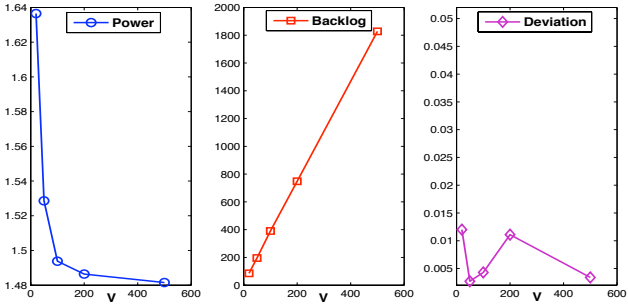}
\caption{LEFT: average network power consumption. MIDDLE: average network backlog size. RIGHT: percentage of time when $\exists\,\,q_j$ such that $|q_j-\gamma^*_{Vj}|>2[\log(V)]^2$.}\label{fig:LIFO-power-backlog}
\end{figure}

Fig. \ref{fig:LIFO_FIFO_stat2}  compares the delay statistics of LIFO and FIFO for more than $99.9\%$ of the packets that leave the system before the simulation ends, under the cases $V=100$ and $V=500$. We see that LIFO not only dramatically reduces the average packet delay for these packets, but also greatly reduces the delay for most of these packets. For instance, when $V=500$, under FIFO, almost all packets experience the average delay around $1220$ slots. Whereas under LIFO, the average packet delay is brought down to $78$. Moreover, $52.9\%$ of the packets only experience delay less than $20$ slots, and $90.4\%$ of the packets experience delay less than $100$ slots. \emph{Hence most packets' delay are reduced by a factor of $12$ under LIFO as compared to that under FIFO! }

\begin{figure}[cht]
\centering
\includegraphics[height=1.1in, width=3.4in]{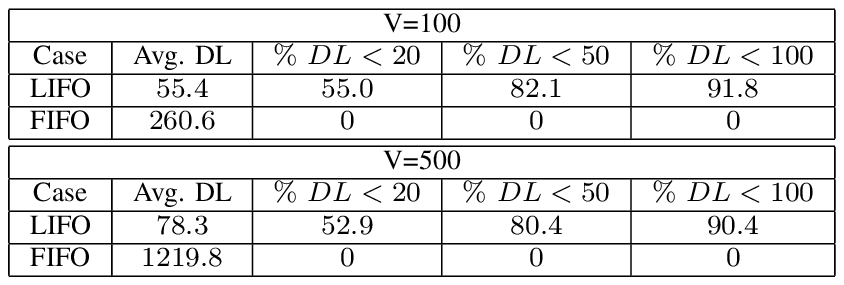}
\caption{Delay statistics under Backpressure with LIFO and FIFO for packets that leave the system before simulation ends (more than $99.9\%$). $\% DL<a$ is the percentage of packets that enter the network and has delay less than $a$. }\label{fig:LIFO_FIFO_stat2}
\end{figure}


Fig.  \ref{fig:LIFO_delay} also shows the delay for the first $20000$ packets that enter the network in the case when $V=500$. We see that under Backpressure plus LIFO, most of the packets experience very small delay; while under Backpressure with FIFO, each packet experiences roughly the average delay.

\begin{figure}[cht]
\centering
\includegraphics[height=2in, width=3.3in]{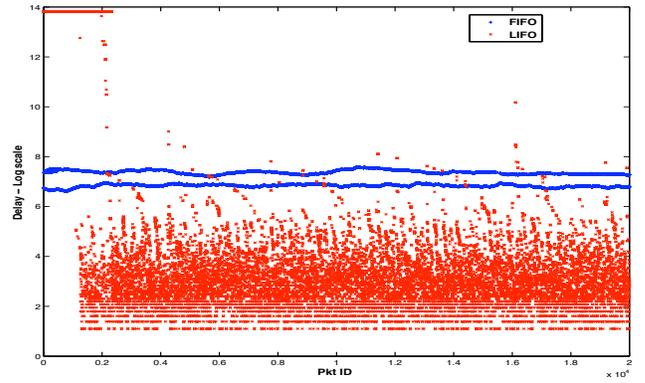}
\caption{Packet Delay under Backpressure with LIFO and FIFO}\label{fig:LIFO_delay}
\end{figure}






\vspace{-.05in}

\section{Empirical Validation}\label{section:experiment}
In this section we validate our analysis empirically by carrying out new experiments over the same testbed and Backpressure Collection Protocol (BCP) code of  \cite{scott_lifo_ipsn}.  This prior work did not empirically evaluate the relationship between $V$, finite storage availability, packet latency and packet discard rate.  We note that BCP runs atop the default CSMA MAC for TinyOS which is not known to be throughput optimal, that the testbed may not precisely be defined by a finite state Markovian evolution, and finally that limited storage availability on real wireless sensor nodes mandates the introduction of virtual queues to maintain backpressure values in the presence of data queue overflows.

In order to avoid using very large data buffers, in \cite{scott_lifo_ipsn} the forwarding queue of BCP has been implemented as a \emph{floating queue}.  The concept of a floating queue is shown in Figure \ref{fig:floatingQueue}, which  operates with a finite data queue of size $D_{max}$ residing atop a virtual queue which preserves backpressure levels. Packets that arrive to a full data queue result in a data queue discard and the incrementing of the underlying virtual queue counter. Underflow events (in which a virtual backlog exists but the data queue is empty) results in null packet generation, which are filtered and then discarded by the destination. 

Despite these real-world differences, we are able to demonstrate clear order-equivalent delay gains due to LIFO usage in BCP in the following experimentation.


\begin{figure}[cht]
\centering
\includegraphics[height=1.2in, width=3.1in]{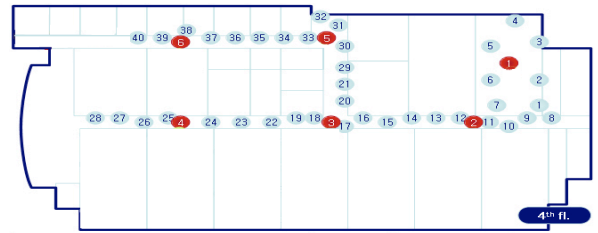}
\caption{The 40 tMote Sky devices used in experimentation on Tutornet.}\label{fig:tutornet}
\end{figure}

\vspace{-.2in}
\subsection{Testbed and General Setup}
To demonstrate the empirical results, we deployed a collection scenario across 40 nodes within the Tutornet testbed (see Figure \ref{fig:tutornet}).  This deployment consisted of Tmote Sky devices embedded in the 4th floor of Ronald Tutor Hall at the University of Southern California.  

In these experiments, one sink mote (ID 1 in Figure \ref{fig:tutornet}) was designated and the remaining 39 motes sourced traffic simultaneously, to be collected at the sink.  The Tmote Sky devices were programmed to operate on 802.15.4 channel 26, selected for the low external interference in this spectrum on Tutornet.  Further, the motes were programmed to transmit at -15 dBm to provide reasonable interconnectivity.  These experimental settings are identical to those used in  \cite{scott_lifo_ipsn}.


\begin{figure}[htbp]
    \centering
\includegraphics[totalheight=0.18\textheight]{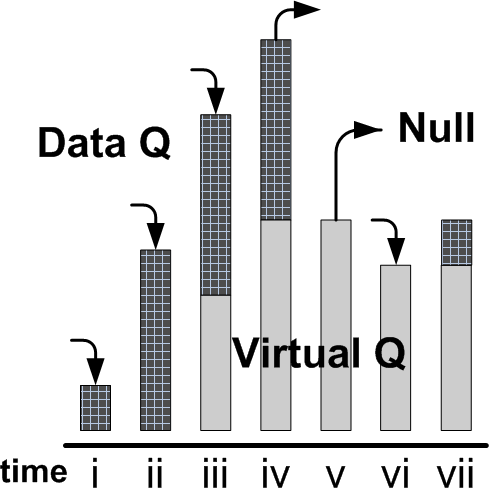}
\caption{\small{The floating LIFO queues of \cite{scott_lifo_ipsn} drop from the data queue during overflow, placing the discards within an underlying virtual queue. Services that cause data queue underflows generate null packets, reducing the virtual queue size.}}
\label{fig:floatingQueue}
\end{figure}

We vary $D_{max}$ over experimentation. In practice, BCP defaults to a $D_{max}$ setting of $12$ packets, the maximum reasonable resource allocation for a packet forwarding queue in these highly constrained devices. 



\subsection{Experiment Parameters}

Experiments consisted of Poisson traffic at 1.0 packets per second per source for a duration of 20 minutes.  This source load  is moderately high, as the boundary of the capacity region for BCP running on this subset of motes has previously been documented at 1.6 packets per second per source \cite{scott_lifo_ipsn}.  A total of 36 experiments were run using the standard BCP LIFO queue mechanism, for all combinations of $V \in \{1,2,3,4,6,8,10,12\}$ and LIFO storage threshold $D_{max} \in \{2,4,8,12\}$.  In order to present a delay baseline for Backpressure we additionally modified the BCP source code and ran experiments with 32-packet FIFO queues (no floating queues) for $V \in \{1,2,3\}$. \footnote{These relatively small $V$ values are due to the constraint that the motes have small data buffers. Using larger $V$ values will cause buffer overflow at the motes. }

\subsection{Results}

Testbed results in Figure \ref{fig:100ppsNullScalingDlyT} provide the system average packet delay from source to sink over $V$ and $D_{max}$, and includes 95\% confidence intervals.  Delay in our FIFO implementation scales linearly with V, as predicted by the analysis in \cite{huangneely_dr_tac}.  This yields an average delay that grows very rapidly with $V$, already greater than 9 seconds per packet for $V=3$. Meanwhile, the LIFO floating queue of BCP performs much differently.  We have plotted a scaled $[\log(V)]^2$ target, and note that as $D_{max}$ increases the average packet delay remains bounded by $\Theta([\log(V)]^2)$.

\begin{figure}[htbp]
\includegraphics[totalheight=0.25\textheight]{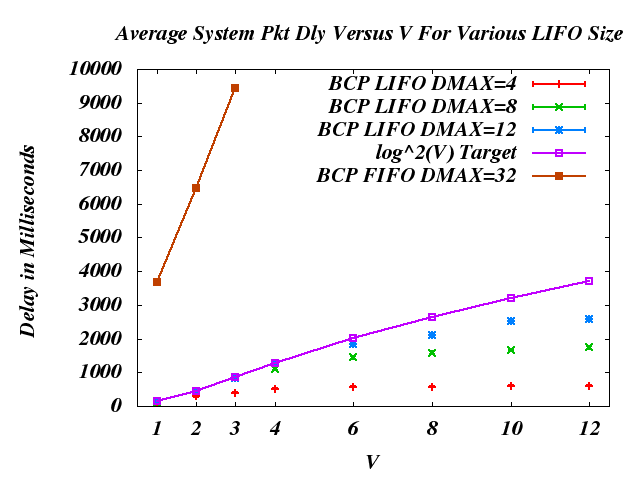}
\caption{\small{System average source to sink packet delay for BCP FIFO versus BCP LIFO implementation over various V parameter settings.}}
\label{fig:100ppsNullScalingDlyT}
\end{figure} 

These delay gains are only possible as a result of discards made by the LIFO floating queue mechanism that occur when the queue size fluctuates beyond the capability of the finite data queue to smooth.  Figure \ref{fig:100ppsNullScalingT} gives the system packet loss rate of BCP's LIFO floating queue mechanism over $V$.  Note that the poly-logarithmic delay performance of Figure \ref{fig:100ppsNullScalingDlyT} is achieved even for data queue size 12, which itself drops at most 5\% of traffic at $V=12$.  
We cannot state conclusively from these results that the drop rate scales like $O(\frac{1}{V^{c_0\log(V)}})$.  We hypothesize that a larger $V$ value would be required in order to observe the predicted drop rate scaling.  Bringing these results back to real-world implications, note that BCP (which minimizes a penalty function of packet retransmissions) performs very poorly with $V=0$, and was found to have minimal penalty improvement for V greater than 2.  At this low V value, BCP's 12-packet forwarding queue demonstrates zero packet drops in the results presented here.  These experiments, combined with those of \cite{scott_lifo_ipsn} suggest strongly that the drop rate scaling may be inconsequential in many real world applications.

\begin{figure}[htbp]
\includegraphics[totalheight=0.25\textheight]{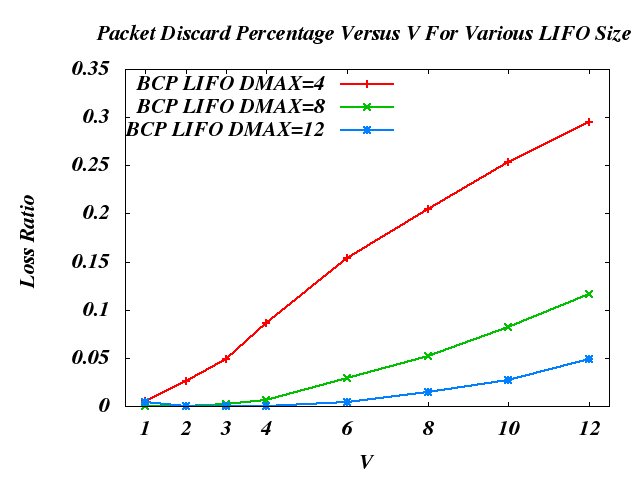}
\caption{\small{System packet loss rate of BCP LIFO implementation over various V parameter settings.}}
\label{fig:100ppsNullScalingT}
\end{figure}

In order to explore the queue backlog characteristics and compare with our analysis, Figure \ref{fig:100ppsQueueHist} presents a histogram of queue backlog frequency for rear-network-node 38 over various V settings.  This node was observed to have the worst queue size fluctuations among all thirty-nine sources.  For $V=2$, the queue backlog is very sharply distributed and fluctuates outside the range $[11-15]$ only 5.92\% of the experiment.  As $V$ is increased, the queue attraction is evident.  For $V=8$ we find that the queue deviates outside the range $[41-54]$ only 5.41\% of the experiment.  The queue deviation is clearly scaling sub-linearly, as a four-fold increase in $V$ required only a 2.8 fold increase in $D_{max}$ for comparable drop performance. 



\begin{figure}[htbp]
\includegraphics[totalheight=0.26\textheight]{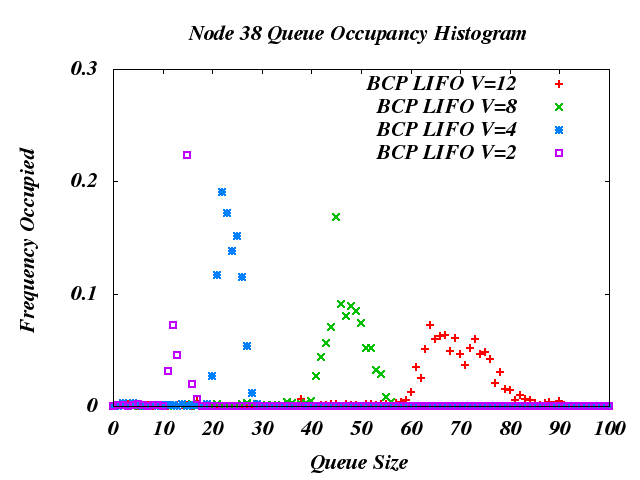}
\caption{\small{Histogram of queue backlog frequency for rear-network-node 38 over various V settings.}}
\label{fig:100ppsQueueHist}
\end{figure}

\vspace{-.2in}
\section{Optimizing Functions of Time averages}\label{section:opt-timeavg}
So far we have focused on optimizing time averages of functions, we now consider the case  when the objective of the network controller is to optimize a function of some time average metric, e.g.,  \cite{neelyfairness}.  
Specifically, we assume that the action $x(t)$ at time $t$ incurs some instantaneous \emph{network attribute vector} $\bv{y}(t)=\bv{y}(x(t))=(y_1(t), ..., y^K(t))^T\in\mathbb{R}_+^K$, and the objective of the network controller is to minimize a cost function $\text{Cost}(\overline{\bv{y}(t)}): \mathbb{R}_+^K\rightarrow\mathbb{R}_+$, \footnote{The case for maximizing a utility  function of long term averages can be treated in a similar way.} where $\overline{\bv{y}(t)}$ represents the time average value of $\bv{y}(t)$. 
We assume that the function $\text{Cost}(\cdot)$ is continuous, convex and is component-wise increasing, and that $|y_k(x(t))|\leq\delta_{max}$ for all $k$, $x(t)$. In this case, we see that the Backpressure algorithm in Section  \ref{section:qlareview}  cannot be directly applied and the deterministic problem (\ref{eq:primal}) also needs to be modified. 

To tackle this problem using the Backpressure algorithm, we introduce an  \emph{auxiliary vector} $\bv{z}(t)=(z_1(t), ..., z_K(t))^T$. 
We then define the following virtual queues $H_k(t), j=1,...,K$ that evolves as follows:
\begin{eqnarray}
H_k(t+1)=\max\big[ H_k(t) - y_k(t), 0 \big] + z_k(t). \label{eq:vq_def}
\end{eqnarray}
These virtual queues are introduced for ensuring that the average value of $y_k(t)$ is no less than the average value of $z_k(t)$. 
We will then try to maximize the time average of the function $\text{Cost}(\bv{z}(t))$, subject to the constraint that the actual queues $q_j(t)$ and the virtual queues $H_k(t)$ must all be stable. Specifically, the Backpressure algorithm for this problem works as follows: 


\underline{\emph{Backpressure:}} At every time slot $t$, observe the current network state $S(t)$, and the backlogs $\bv{q}(t)$ and $\bv{H}(t)$. If $S(t)=s_i$, do the following: 
\begin{enumerate}
\item \underline{Auxiliary vector:} choose the vector $\bv{z}(t)$ by solving: 
\begin{eqnarray}
\hspace{-.3in}\min: && V\text{Cost}(\bv{z})+\sum_{k}H_k(t)z_k\nonumber\\
\hspace{-.3in}s.t. && 0\leq z_k\leq\delta_{max}.\label{eq:qla-auxi}
\end{eqnarray}

\item \underline{Action:} choose the action $x(t)\in\script{X}^{(s_i)}$ that solves:
\begin{eqnarray}
\hspace{-.5in}&&\max: \,\, \sum_{k}H_k(t)y_k(x)+\sum_jq_j(t)[\mu_j(s_i, x)- A_j(s_i, x)]\nonumber\\
\hspace{-.5in}&&\quad s.t. \,\,\,\, x\in\script{X}^{(s_i)}.\label{eq:qla-actual}
\end{eqnarray}
\end{enumerate}
In this case, one can also show that this Backpressure algorithm achieves the $[O(1/V), O(V)]$ utility-delay tradeoff under a Markovian $S(t)$ process. 
We also note that in this case, the deterministic problem is slightly different. Indeed, the intuitively formulation will be of the following form: 
\begin{eqnarray}
\min:&&\script{F}(\bv{x})\triangleq V\text{Cost}(\sum_{s_i}\pi_{s_i}\bv{y}(x^{(s_i)}))\label{eq:primal_func}\\
s.t.&&\script{A}_j(\bv{x})\triangleq\sum_{s_i}\pi_{s_i}A_j(s_i, x^{(s_i)})\nonumber\\
&&\qquad\qquad\quad\leq \script{B}_j(\bv{x})\triangleq\sum_{s_i}\pi_{s_i}\mu_j(s_i, x^{(s_i)})\quad\forall\, j\nonumber\\
&& x^{(s_i)}\in \script{X}^{(s_i)}\quad \forall\, i=1, 2, ..., M.\nonumber
\end{eqnarray}
However, the dual problem of this optimization problem is not separable, i.e., not of the form of (\ref{eq:dual_separable}), unless the function $\text{Cost}(\cdot)$ is linear or if there exists an optimal action that is in every feasible action set $\script{X}^{(s_i)}$, e.g.,  \cite{neelyfairness}. 
To get rid of this problem, we introduce the auxiliary vector $\bv{z}=(z_1, ..., z_K)^T$ and change the problem to:
\begin{eqnarray}
\min:&&\script{F}(\bv{x})\triangleq V\text{Cost}(\bv{z})\label{eq:primal_func_2}\\
s.t. &&  \qquad\bv{z}\preceq \sum_{s_i}\pi_{s_i}\bv{y}(x^{(s_i)})\nonumber\\
&&\script{A}_j(\bv{x})\triangleq\sum_{s_i}\pi_{s_i}A_j(s_i, x^{(s_i)})\nonumber\\
&&\qquad\qquad\quad\leq \script{B}_j(\bv{x})\triangleq\sum_{s_i}\pi_{s_i}\mu_j(s_i, x^{(s_i)})\quad\forall\, j\nonumber\\
&& x^{(s_i)}\in \script{X}^{(s_i)}\quad \forall\, i=1, 2, ..., M.\nonumber
\end{eqnarray}
It can be shown that this modified problem is equivalent to (\ref{eq:primal_func}). 
Now we see that it is indeed due to the non-separable feature of (\ref{eq:primal_func}) that we need to introduce the auxiliary vector $\bv{z}(t)$ in the Backpressure problem. 
We also  note that the problem (\ref{eq:primal_func_2}) actually has the form of (\ref{eq:primal}). Therefore, all previous results on (\ref{eq:primal}), e.g., Theorem \ref{thm:prob_multi_con} and \ref{thm:qlalifo_delay} will also apply to problem (\ref{eq:primal_func_2}). 

\vspace{-.2in}
\section*{Appendix A -- Proof of \ref{theorem:little1} }
Here we provide the proof of Theorem \ref{theorem:little1}. 
\begin{proof} 
Consider a sample path for which the $\liminf$ arrival rate is at least $\lambda_{min}$ and for which we have an infinite
number of departures (this happens with probability 1 by assumption).  
There must be a non-empty subset of $\script{B}$ consisting of buffer locations that experience an infinite number of departures. Call this subset $\tilde{\script{B}}$.  Now let $W^{(b)}_i$ be the delay of the $i^{th}$ departure from $b$, 
let $D^{(b)}(t)$ denote the number of departures from a buffer slot $b\in\tilde{\script{B}}$ up to time $t$, and use $Q^{(b)}(t)$ to denote the occupancy of the buffer slot $b$ at time $t$. Note that $Q^{(b)}(t)$ is either $0$ or $1$. 
For all $t\geq0$, it can be shown that: 
\begin{eqnarray}
\sum_{i=1}^{D^{(b)}(t)} W_i^{(b)} \leq \int_0^t Q^{(b)}(\tau)d\tau. \label{eq:Little-sample-path}
\end{eqnarray}
This can be seen from Fig. \ref{fig:littles-thm-timeline} below. 
\begin{figure}[htbp]
   \centering
   \includegraphics[height=.6in, width=2.6in]{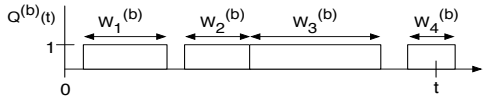} 
   \caption{An illustration of inequality (\ref{eq:Little-sample-path}) for a particular buffer location $b$.  
   At time $t$ in the figure, we have $D^{(b)}(t) = 3$.}
   \label{fig:littles-thm-timeline}
\end{figure}
\vspace{-.1in}

Therefore, we have:
\begin{eqnarray} 
 \sum_{b \in \tilde{\script{B}}} \sum_{i=1}^{D^{(b)}(t)} W_i^{(b)} &\leq&\int_0^{t}  \sum_{i\in\tilde{\script{B}}}Q^{(b)}(\tau)d\tau \nonumber\\
 &\leq& \int_0^{t} |\tilde{\script{B}}|d\tau \nonumber\\
 &\leq& |\script{B}|t.  \label{eq:lp1} 
 \end{eqnarray} 
 
The left-hand-side of the above inequality is equal to the sum of all delays of jobs that depart from locations in $\tilde{\script{B}}$ 
up to time $t$.  All other buffer locations (in $\script{B}$ but not in $\tilde{\script{B}}$)
experience only a finite number of departures.  Let $\script{J}$ represent an
index set that indexes all of the (finite number) of jobs that depart from these other locations. Note that the delay $W_j$
for each job $j\in\script{J}$ is finite (because, by definition, job $j$ eventually departs).   We thus have: 
\[ \sum_{i=1}^{D(t)} W_i \leq \sum_{b \in \tilde{\script{B}}} \sum_{i=1}^{D^{(b)}(t)} W_i^{(b)} + \sum_{j\in\script{J}} W_j. \]
where the inequality is because the second term on the right-hand-side 
sums over jobs in $\script{J}$, and these jobs 
may not have departed by time $t$. 
Combining the above and (\ref{eq:lp1}) yields for all $t\geq0$: 
\[  \sum_{i=1}^{D(t)} W_i \leq |\script{B}|t + \sum_{j\in\script{J}} W_j.   \]
Dividing by $D(t)$ yields: 
\[ \frac{1}{D(t)}\sum_{i=1}^{D(t)} W_i \leq \frac{|\script{B}|t}{D(t)} + \frac{1}{D(t)}\sum_{j\in\script{J}} W_j.  \]
Taking a $\limsup$ as $t\rightarrow\infty$ yields: 
\begin{equation} \label{eq:lp2} 
\limsup_{t\rightarrow\infty} \frac{1}{D(t)}\sum_{i=1}^{D(t)} W_i  \leq \limsup_{t\rightarrow\infty}  \frac{|\script{B}|t}{D(t)},    
\end{equation} 
where we have used the fact that $\sum_{j\in\script{J}} W_j$ is a finite number, and $D(t)\rightarrow\infty$ as $t\rightarrow\infty$, so that: 
\vspace{-.05in}
\[ \limsup_{t\rightarrow\infty} \frac{1}{D(t)}\sum_{j\in\script{J}} W_j  = 0.  \]

Now note that, because each buffer location in $\script{B}$ can hold at most one job, the number of departures $D(t)$ is 
at least $N(t) - |\script{B}|$, which is a positive number for sufficiently large $t$. Thus: 
\begin{eqnarray*} 
\limsup_{t\rightarrow\infty} \frac{|\script{B}|t}{D(t)} &\leq& \limsup_{t\rightarrow\infty} \left[\frac{|\script{B}|t}{N(t) - |\script{B}|}\right]  \\
&=& \limsup_{t\rightarrow\infty} \left[\frac{|\script{B}|}{N(t)/t - |\script{B}|/t}\right] \\
&\leq& |\script{B}|/\lambda_{min}. 
\end{eqnarray*}
Using this in (\ref{eq:lp2}) proves the result.  
\end{proof}

$\vspace{-.2in}$
\bibliographystyle{unsrt}
\bibliography{mybib}

\end{document}